\documentclass[a4paper,12pt]{amsart}
\usepackage{amssymb}
\usepackage{ifthen}
\usepackage{graphicx}
\usepackage{float}
\usepackage{caption}
\usepackage{subcaption}
\usepackage{cite}
\usepackage{amsfonts}
\usepackage{amscd}
\usepackage{amsxtra}
\usepackage{color}
\newtheorem{thm}{Theorem}
\newtheorem{cor}{Corollary}
\newtheorem{lem}{Lemma}
\newtheorem{rem}{Remark}
\newtheorem{prop}{Proposition}

\newtheorem{conj}{Conjecture}
\newtheorem{prob}{Problem}

\theoremstyle{definition}
\newtheorem{defn}{Definition}[section]
\newtheorem{example}{Example}

\newenvironment{pf}[1][]{%
 \vskip 1mm
 \noindent
 \ifthenelse{\equal{#1}{}}%
  {{\slshape Proof. }}%
  {{\slshape #1.} }%
 }%
{\qed\bigskip}

\newcounter{alphabet}
\newcounter{tmp}
\newenvironment{Thm}[1][]{\refstepcounter{alphabet}%
\bigskip%
\noindent%
{\bf Theorem \Alph{alphabet}}%
\ifthenelse{\equal{#1}{}}{}{ (#1)}%
{\bf .} \itshape}{\vskip 8pt}

\makeatletter
\newcommand{\Ref}[1]{\@ifundefined{r@#1}{}{\setcounter{tmp}{\ref{#1}}\Alph{tmp}}}
\makeatother

\newenvironment{Lem}[1][]{\refstepcounter{alphabet}%
\bigskip%
\noindent%
{\bf Lemma \Alph{alphabet}}%
{\bf .} \itshape}{\vskip 8pt}

\newcommand{\IC}{{\mathbb C}}
\newcommand{\ID}{{\mathbb D}}

\def\be{\begin{equation}}
\def\ee{\end{equation}}

\newcommand{\bee}{\begin{enumerate}}
\newcommand{\eee}{\end{enumerate}}

\newcommand{\blem}{\begin{lem}}
\newcommand{\elem}{\end{lem}}
\newcommand{\bthm}{\begin{thm}}
\newcommand{\ethm}{\end{thm}}
\newcommand{\bcor}{\begin{cor}}
\newcommand{\ecor}{\end{cor}}
\newcommand{\beg}{\begin{example}}
\newcommand{\eeg}{\end{example}}
\newcommand{\begs}{\begin{examples}}
\newcommand{\eegs}{\end{examples}}
\newcommand{\bdefe}{\begin{defn}}
\newcommand{\edefe}{\end{defn}}
\newcommand{\bprob}{\begin{prob}}
\newcommand{\eprob}{\end{prob}}
\newcommand{\bques}{\begin{ques}}
\newcommand{\eques}{\end{ques}}
\newcommand{\bei}{\begin{itemize}}
\newcommand{\eei}{\end{itemize}}
\newcommand{\bcon}{\begin{conj}}
\newcommand{\econ}{\end{conj}}
\newcommand{\bcons}{\begin{conjs}}
\newcommand{\econs}{\end{conjs}}
\newcommand{\bprop}{\begin{propo}}
\newcommand{\eprop}{\end{propo}}
\newcommand{\br}{\begin{rem}}
\newcommand{\er}{\end{rem}}
\newcommand{\brs}{\begin{rems}}
\newcommand{\ers}{\end{rems}}
\newcommand{\bo}{\begin{obser}}
\newcommand{\eo}{\end{obser}}
\newcommand{\bos}{\begin{obsers}}
\newcommand{\eos}{\end{obsers}}
\newcommand{\bpf}{\begin{pf}}
\newcommand{\epf}{\end{pf}}
\newcommand{\ba}{\begin{array}}
\newcommand{\ea}{\end{array}}
\newcommand{\beq}{\begin{eqnarray}}
\newcommand{\beqq}{\begin{eqnarray*}}
\newcommand{\eeq}{\end{eqnarray}}
\newcommand{\eeqq}{\end{eqnarray*}}

\newcommand{\ds}{\displaystyle}

\newcounter{minutes}
\divide\time by 60
\newcounter{hours}
\multiply\time by 60 \addtocounter{minutes}{-\time}

\begin{document}
\bibliographystyle{amsplain}
\title[Improved Bohr's inequality for locally univalent harmonic mappings]{Improved Bohr's inequality for  locally univalent harmonic mappings}

\thanks{
File:~\jobname .tex,
          printed: \number\day-\number\month-\number\year,
          \thehours.\ifnum\theminutes<10{0}\fi\theminutes}


\author{Stavros Evdoridis, Saminathan Ponnusamy and Antti Rasila}


\address{S. Evdoridis, Aalto University, Department of Mathematics and Systems Analysis, P. O. Box 11100, FI-00076 Aalto, Finland.}
\email{stavros.evdoridis@aalto.fi}

\address{S. Ponnusamy, Stat-Math Unit,
Indian Statistical Institute (ISI), Chennai Centre,
110, Nelson Manickam Road,
Aminjikarai, Chennai, 600 029, India.}
\email{samy@isichennai.res.in, samy@iitm.ac.in}

\address{A. Rasila, Aalto University, Department of Mathematics and Systems Analysis, P. O. Box 11100, FI-00076 Aalto, Finland.}
\email{antti.rasila@aalto.fi}


\begin{abstract}
We prove several improved versions of Bohr's inequality for the harmonic mappings of the form $f=h+\overline{g}$, where $h$ is bounded by 1 and $|g'(z)|\le|h'(z)|$.  The improvements are obtained along the lines of an earlier work of Kayumov and Ponnusamy, i.e. \cite{KayPon2}, for example a term related to the area of the image of the disk $D(0,r)$ under the mapping $f$ is considered. Our results are sharp. In addition, further improvements of the main results for certain special classes of harmonic mappings are provided.
\end{abstract}


\maketitle
\pagestyle{myheadings}
\markboth{S. Evdoridis, S. Ponnusamy and  A. Rasila}{Bohr's inequality for locally univalent harmonic mappings}

\section{Introduction and Preliminaries}\label{KayPon8-sec1}
In  recent years there has been considerable interest in the classical inequality of Bohr \cite{Bohr-14}, which states that
if $f$ is a bounded analytic function on the unit disk $\ID :=\{z\in\IC:\, |z|<1\}$, with the Taylor expansion $f(z)=\sum_{k=0}^\infty a_k z^k$,
then
$$ \sum_{k=0}^\infty |a_k|r^k\leq \|f\|_\infty ~\mbox{ for $|z|=r\leq 1/3$},
$$
and the constant $1/3$ is sharp. Bohr obtained this inequality only for $r\leq {1}/{6}$, but later M.~Riesz, I.~Schur and N.~ Wiener independently proved
its validity for $r\leq {1}/{3}$. Several other proofs were also available in the literature. Moreover, similar problems were considered for Hardy spaces or for more abstract spaces.
For background information about this inequality and further work related to Bohr's inequality, we refer the reader to
the recent survey by Abu-Muhanna et al. \cite{AAPon1} and the references therein.
Other recent results on this topic include \cite{AliBarSoly,KayPon1,KayPon1b, KayPon2,KayPon3}. Harmonic version of Bohr's inequality was discussed by Kayumov et al. in \cite{KayPon3}.
In a related development, Kayumov and Ponnusamy \cite{KayPon2} gave several improved versions of Bohr's inequality. Some of them may now be recalled.

\begin{Thm}\label{KayPon8-Additional4}
Suppose that $f(z) = \sum_{k=0}^\infty a_k z^k$ is analytic in $\ID$, $|f(z)| \leq 1$ in $\ID$ and  $S_r$ denotes the area of the
image of the subdisk $|z|<r$ under the mapping $f$. Then
\begin{equation}\label{Eq_Th3}
B_1(r):=\sum_{k=0}^\infty |a_k|r^k+\frac{16}{9}\left (\frac{S_r}{\pi}\right )  \leq 1 ~\mbox{ for  }~ r \leq \frac{1}{3},
\end{equation}
and the constants $1/3$ and $16/9$ cannot be improved. Moreover,
\begin{equation}\label{Eq2_Th3}
B_2(r):=|a_0|^2+\sum_{k=1}^\infty |a_k|r^k+\frac{9}{8}\left (\frac{S_r}{\pi}\right )  \leq 1 ~\mbox{ for  }~ r \leq \frac{1}{2},
\end{equation}
and the constants $1/2$ and $9/8$ cannot be improved.
\end{Thm}

\begin{Thm}\label{KayPon8-1to3}
Suppose that $f(z) = \sum_{k=0}^\infty a_k z^k$ is analytic in $\ID$ and $|f(z)| \leq 1$ in $\ID$. Then we have
\begin{enumerate}
\item $\ds
|a_0|+\sum_{k=1}^\infty \left(|a_k|+\frac{1}{2}|a_k|^2\right)r^k  \leq 1 ~\mbox{ for  }~ r \leq \frac{1}{3},
$
and the constants $1/3$ and $1/2$  cannot be improved.
\item $\ds \sum_{k=0}^\infty |a_k|r^k + |f(z)-a_0|^2 \leq 1 ~\mbox{ for  }~ r \leq \frac{1}{3}, $
and the constant $1/3$  cannot be improved.
\item $\ds
|f(z)|^2+\sum_{k=1}^\infty |a_k|^2r^{2k} \leq 1 ~\mbox{ for  }~ r \leq \sqrt{\frac{11}{27}},
$
and the constant $11/27$ cannot be improved.
\end{enumerate}
\end{Thm}

%
%
%
%

The primary objective of this article is to obtain harmonic versions of the above results and thereby we present improved versions of certain main
results of Kayumov et al. \cite{KayPon3} in the case of harmonic mappings. For this aim, we need to introduce some basic results on harmonic mappings.

A complex-valued function $f=u+iv$ defined on $\ID$ is harmonic if $u$ and $v$ are real-harmonic in $\ID$. Every harmonic function $f$ admits
the canonical representation $f=h+\overline{g}$, where $h$ and $g$ are analytic in $\mathbb{D}$ such that $g(0)=0=f(0)$. A locally univalent harmonic function $f$ in $\ID$ is said to be
sense-preserving if the Jacobian $J_{f}(z)$, $J_{f}(z)= |h'(z)|^{2} -|g'(z)|^{2}$, is positive in ${\mathbb{D}}$
or equivalently, its  dilatation $\omega_f(z)=g'(z)/h'(z)$ satisfies the inequality $|\omega_f(z)|<1$ for $z\in \ID$ (see \cite{Le} and \cite{CS,dur,PonRasi2013}).
Properties of harmonic mappings have been investigated extensively, especially after the
appearance of the pioneering work of Clunie and Sheil-Small \cite{CS} in 1984. A comprehensive reference on this topic is the monograph of Duren \cite{dur}.

The proofs of the main results rely on lemmas that are stated below.

\begin{Lem} {\rm(\cite[Proof of Theorem 1]{KayPon1} and \cite{KayPon1b})}
Suppose that $f(z) = \sum_{k=0}^\infty a_k z^k$ is analytic in $\ID$, $|f(z)| \leq 1$ in $\ID$. Then we have
\be\label{KayPon8-eq7}
\sum_{k=1}^\infty |a_k|r^k \leq
\left \{ \begin{array}{lr} \ds A(r):=r\frac{1-|a_0|^2}{1-r|a_0|} & \mbox{ for $|a_0| \ge r,$} \\[4mm]
\ds B(r):=r\frac{\sqrt{1-|a_0|^2}}{\sqrt{1-r^2}} &\mbox{ for $|a_0| < r$}.
\end{array} \right .
\ee
\end{Lem}

Our next result concerns sense-preserving harmonic mappings defined on   $\ID$.

\begin{Lem}\label{KayPon9-thm2}
{\rm (\cite{KayPon3})}
Suppose that $f(z) = h(z)+\overline{g(z)}=\sum_{k=0}^\infty a_k z^k+\overline{\sum_{k=1}^\infty b_k z^k}$ is a
harmonic mapping of the disk $\ID$, where $h$ is a bounded function in $\ID$ and $|g'(z)|\leq |h'(z)|$ for $z\in\ID$
(the later condition obviously holds if $f$ is sense-preserving). Then  for each $r\in [0,1)$,
\be\label{KP9-eq6}
\sum_{k=1}^\infty |b_k|^2r^k \leq \sum_{k=1}^\infty |a_k|^2r^k .
\ee
\end{Lem}

\begin{Lem}\label{KP2-lem2}
{\rm (\cite[Lemma 1]{KayPon2})}
If $h(z)=\sum_{k=0}^{\infty} a_kz^k$ is analytic and satisfies the inequality $|h(z)| <  1$ in $\ID$, then
the following sharp inequality holds:
\begin{equation}\label{KP2-eq3}
\sum_{k=1}^\infty k |a_k|^2r^{k} \leq r\frac{(1-|a_0|^2)^2}{(1-|a_0|^2r)^2} ~\mbox{  for $0 < r \leq 1/2$}.
\end{equation}
\end{Lem}

\section{Main Results}

\begin{thm}\label{HKayPon8}
Suppose that $f(z) = h(z)+\overline{g(z)}=\sum_{k=0}^\infty a_k z^k+\overline{\sum_{k=1}^\infty b_k z^k}$ is a
harmonic mapping of the disk $\ID$, where $h$ is a bounded function in $\ID$ such that $|h(z)|<1$ and $|g'(z)|\leq |h'(z)|$ for $z\in\ID$.
If $S_r$ denotes the area of the image of the subdisk $|z|<r$ under the mapping $f$, then
\begin{equation}\label{HEq_Th3}
H_1(r):=|a_0|+\sum_{k=1}^\infty (|a_k| +|b_k|)r^k+ \frac{108}{25}\left (\frac{S_r}{\pi}\right )  \leq 1 ~\mbox{ for  }~ r \leq \frac{1}{5},
\end{equation}
and the constants $1/5$ and $c=108/25$ cannot be improved.   Moreover,
\begin{equation}\label{HEq2_Th3}
H_2(r):=|a_0|^2+\sum_{k=1}^\infty (|a_k| +|b_k|)r^k+\frac{4}{3}\left (\frac{S_r}{\pi}\right )  \leq 1 ~\mbox{ for  }~ r \leq  \frac{1}{3},
\end{equation}
and the constants $1/3$ and $4/3$ cannot be improved.
\end{thm}
\begin{proof}
By assumption $f = h+\overline{g}$ is harmonic in $\ID$, where $h(z)=\sum_{k=0}^\infty a_k z^k$ is such that $|h(z)|< 1 $ for $z\in\ID$.
Then it is well-known that $|a_k| \leq 1-|a_0|^2$ for all $k\geq 1$.

Firstly, for the area of the image of the disk $|z|<r$ under the mapping $f$, we have
\begin{eqnarray*}
\frac{S_r}{\pi} &=& \frac{1}{\pi}\iint_{|z|<r} \left ( |h'(z)|^2 - |g'(z)|^2\right )dxdy\\
&=&  \sum_{k=1}^\infty k(|a_k|^2-|b_k|^2)r^{2k} \\
&\leq &   \sum_{k=1}^\infty k(1-|a_0|^2)^2r^{2k}
=  (1-|a_0|^2)^2 \frac{r^2}{(1-r^2)^2}.
\end{eqnarray*}
In particular, it follows that
\begin{eqnarray}
\frac{S_{1/5}}{\pi} \leq \frac{25}{24^2}(1-|a_0|^2)^2 \label{Th1-ex1}
\end{eqnarray}
and
\begin{eqnarray}\label{Th1-ex1a}
\frac{S_{1/3}}{\pi} \leq \frac{9}{64}(1-|a_0|^2)^2.
\end{eqnarray}
We see that $H_1$ given by \eqref{HEq_Th3} is an increasing function of $r$ and hence it suffices to prove \eqref{HEq_Th3} for $r=1/5$.
Moreover, by the condition $|g'(z)| \leq |h'(z)|$, \eqref{KP9-eq6} holds.
Furthermore, by Lemma \Ref{KP2-lem2}  and the hypothesis on $h$, \eqref{KP2-eq3} holds. Dividing by $r$ on both sides of \eqref{KP2-eq3} shows that
\begin{equation}\label{KP2-eq3-a}
\sum_{k=1}^\infty k |a_k|^2r^{k-1} \leq \frac{(1-|a_0|^2)^2}{(1-|a_0|^2r)^2} ~\mbox{  for $0 < r \leq 1/2$}.
\end{equation}
Integrating \eqref{KP2-eq3-a} shows that
\begin{equation}\label{KP2-eq3-b}
\sum_{k=1}^\infty |a_k|^2r^{k} \leq \frac{r(1-|a_0|^2)^2}{1-|a_0|^2r} ~\mbox{ for $0 < r \leq 1/2$},
\end{equation}
which by \eqref{KP9-eq6} yields
$$\sum_{k=1}^\infty |b_k|^2r^k \leq \frac{r(1-|a_0|^2)^2}{1-|a_0|^2r} ~\mbox{ for $0 < r \leq 1/2$.}
$$
Consequently, for $|a_0|<1$ and $0 < r \leq 1/2$, we see that
\be\label{KP2-eq3-b1}
\sum_{k=1}^\infty |b_k|r^k \leq \sqrt{\sum_{k=1}^\infty |b_k|^2r^k}\sqrt{\sum_{k=1}^\infty r^k}
 \leq C(r):=\frac{(1-|a_0|^2)r}{\sqrt{(1-|a_0|^2r)(1-r)}}.
\ee

Now, let $|a_0| \geq 1/5$. Then by \eqref{KayPon8-eq7},  \eqref{Th1-ex1} and \eqref{KP2-eq3-b1}, we have for $r\leq 1/5$ that
\begin{eqnarray}
H_1(r) & \leq & |a_0| + A(1/5) + C(1/5)+ \frac{108}{25}\left( \frac{S_{1/5}}{\pi}\right) \nonumber \\
&\leq & |a_0| +  \frac{1-|a_0|^2}{5-|a_0|}+ \frac{1-|a_0|^2}{2\sqrt{5-|a_0|^2}} +\frac{3}{16}(1-|a_0|^2)^2 \label{Th1_eq1} \\
&=& 1- \frac{1-x}{2(5-x)\sqrt{5-x^2}}\Phi (x), \nonumber
\end{eqnarray}
where $x=|a_0|$, $c=3/16$, and
$$\Phi(x)= 4(2-x)\sqrt{5-x^2} -(1+x)(5-x)-2c(1+x)(1-x^2)(5-x)\sqrt{5-x^2}.
$$
It is easy to see that $\Phi(1)=0$ and
\begin{eqnarray*}
\Phi '(x) &=&  \frac{2 }{\sqrt{5 - x^2}}\Big[4 x^2 + c (-20 + 65 x + 68 x^2 - 38 x^3 - 16 x^4 + 5 x^5) \\
&&+ x (-4 + \sqrt{5 - x^2}) -2 (5 + \sqrt{5 - x^2})\Big].
\end{eqnarray*}
This may be rewritten as
\be\label{KayPon8-eq7a}
\Phi '(x) =  \frac{1}{8\sqrt{5 - x^2}}\Psi (x),
\ee
where
$$\Psi (x)= 268 x^2 - 114 x^3 - 48 x^4 + 15 x^5 - 4 (55 + 8 \sqrt{5 - x^2}) +  x (131 + 16 \sqrt{5 - x^2}).
$$
Now, after some computations, we find that
$$\Psi '(x)=131 + x(536  - 342 x - 192 x^2) + 75 x^4 + \frac{16x(2-x)}{\sqrt{5 - x^2}} + 16 \sqrt{5 - x^2},
$$
which clearly implies that $\Psi '(x)>0$ on $(0,1]$. Thus, $\Psi (x)\leq \Psi (1)=0$
which, by \eqref{KayPon8-eq7a}, shows that $\Phi $ is decreasing on $[0,1]$. Consequently, we have
$$\Phi (x) \geq \Phi (1)=0.$$
Hence, by \eqref{Th1_eq1}, it follows that $H_1(r)\leq 1$ for $1/5\leq |a_0|\leq 1$ and $r\leq 1/5$.


Again, in the case $0\leq |a_0| <1/5$, we may apply \eqref{KayPon8-eq7},  \eqref{Th1-ex1}, \eqref{KP2-eq3-b1} to obtain
\begin{eqnarray*}
H_1(r)& \leq &|a_0| + B(1/5)  + C(1/5) + \frac{108}{25}\left(  \frac{S_{1/5}}{\pi} \right )\\
&\leq & |a_0| + \frac{\sqrt{1-|a_0|^2}}{\sqrt{24}} + \frac{1-|a_0|^2}{2\sqrt{5-|a_0|^2}}+ \frac{3}{16}(1-|a_0|^2)^2\\
& \leq &\frac{1}{5} +\frac{1}{2\sqrt{6}} +\frac{5}{4\sqrt{31}} +\frac{3}{16} <1
\end{eqnarray*}
for $r\leq 1/5$. This proves \eqref{HEq_Th3}.

To see the sharpness of the result we consider the function $f_0=h_0+\overline{g_0}$, where
\be\label{KayPon8-eq7c}
h_0(z)= \frac{a+z}{1+\overline{a} z}=a +(|a|^2-1) \sum_{k=1}^\infty \overline{a}^{k-1}z^k
~\mbox{ and }~
g_0(z)=\lambda (|a|^2-1) \sum_{k=1}^\infty \overline{a}^{k-1}z^k ,
\ee
with $a \in \mathbb{D}$ and $\lambda \in \partial\mathbb{D}$. Then both $h_0,g_0$ are analytic in the closed unit disk, and
$g_0'(z)=\lambda h_0'(z)$. If we let
$$D_c(r)=|a_0|+\sum_{k=1}^\infty (|a_k| +|b_k|)r^k+  c\left (\frac{S_r}{\pi}\right ),
$$
then $D_c(r)= H_1(r)$ when $c =108/25$. Thus, in this case, we  compute that
\begin{eqnarray*}
D_c(1/5) &=&  |a|+ \frac{(1+|\lambda |)(1-|a|^2)}{5-|a|} + 25c \frac{(1-|\lambda |^2)(1-|a|^2)^2}{(25-|a|^2)^2} \\
&=& 1-(1-|a|)\left[ 1-\frac{(1+|\lambda |)(1+|a|)}{5-|a|} -25c \frac{(1-|\lambda |^2)(1+|a|)(1-|a|^2)}{(25-|a|^2)^2} \right] .
\end{eqnarray*}
For $c>108/25$ and $|a|$ sufficiently close to $1$, the quantity in the square bracket term is negative and hence, $D_c(1/5)>1$.
Indeed, in this case, we observe that
$$(25-|a|^2)^2 -(1+|\lambda |)(1+|a|)(5+|a|)(25-|a|^2) -25c(1-|\lambda |^2)(1+|a|)(1-|a|^2)$$
is less than
$$(25-|a|^2)^2 -(1+|\lambda |)(1+|a|)(5+|a|)(25-|a|^2) -108(1-|\lambda |^2)(1+|a|)(1-|a|^2) ,$$
which tends to $0$ when both $|a|$ and $|\lambda |$ tend to $1$.

For the second part, namely \eqref{HEq2_Th3}, as $H_2$ is an increasing function of $r$, it is again enough to
prove \eqref{HEq2_Th3} for $r=1/3$. Thus, for $|a_0| \geq 1/3$ and $r\leq 1/3$, as in the previous part by using \eqref{KayPon8-eq7},  \eqref{Th1-ex1a} and \eqref{KP2-eq3-b1},
we deduce that
\begin{eqnarray*}
H_2(r) &\leq & |a_0|^2 + A(1/3) + C(1/3) + \frac{4}{3}\left(\frac{S_{1/3}}{\pi}\right)  \nonumber \\
&\leq & |a_0|^2 + \frac{1-|a_0|^2}{3-|a_0|} +  \frac{1-|a_0|^2}{\sqrt{2(3-|a_0|^2)}}+ \frac{3}{16} (1-|a_0|^2)^2 
\\
&=& 1- (1-|a_0|^2)\left [\frac{2-|a_0|}{3-|a_0|}  -\frac{1}{\sqrt{2(3-|a_0|^2)}} - \frac{3}{16}(1-|a_0|^2)\right ]
\nonumber \\
&=& 1- \frac{1-|a_0|^2}{16(3-|a_0|)\sqrt{2(3-|a_0|^2)}}\Phi _2(|a_0|), \nonumber
\end{eqnarray*}
where
\begin{eqnarray*}
\Phi _2(x) &=& [16(2-x)-3(1-x^2)(3-x)]\sqrt{2(3-x^2)} -16(3-x) \\
&=& (23-13x +9x^2-3x^3)\sqrt{2(3-x^2)} -16(3-x),
\end{eqnarray*}
for $x\in [1/3,1].$ It is easy to see that $\Phi_2(1)=0$ and, by a computation, we have
\be\label{KayPon8-eq7b}
\Phi_2 '(x) = - \frac{2}{\sqrt{2(3 - x^2)}}\Psi_2 (x),  \quad x\in [1/3,1],
\ee
where
$$\Psi_2 (x)= 39-31x +x^2+27 x^3 - 12x^4  - 8\sqrt{2(3 - x^2}).
$$
Now, $16\leq 8\sqrt{2(3 - x^2})<20$ for $x\in [1/3,1]$ and
$$ -31x +x^2+27 x^3 - 12x^4  \geq  -31x +x^3+27 x^3-12x^3= -31x +16 x^3=:\varphi (x),
$$
where $\varphi (x)=-31x +16 x^3$ attains its minimum at $x_0=\sqrt{31/48}\in [1/3,1]$. Consequently,
$$ \min_{x \in [1/3, 1]} (16 x^3 - 31 x) = \varphi (x_0) = \frac{31}{3}\sqrt{\frac{31}{12}}   \approx - 16.6085. $$
Therefore,
$$ \Psi_2 (x) \geq 39- \frac{31}{3}\sqrt{\frac{31}{12}} - 20  \approx 2.39149.
$$
Thus, we see that $\Psi_2 (x) >0$ on $[1/3,1]$ and hence, $\Phi_2 $ is decreasing on $[1/3,1]$ so that
$\Phi_2 (x) \geq \Phi_2 (1)=0.$

For $ |a_0| <1/3$ and $r\leq 1/3$, as before we observe that
\begin{eqnarray*}
H_2(r) &\leq & |a_0|^2 + B(1/3)  + C(1/3) + \frac{4}{3}\left(\frac{S_{1/3}}{\pi}\right)\\
&\leq & |a_0|^2 + \sqrt{\frac{1-|a_0|^2}{8}} + \frac{1-|a_0|^2}{\sqrt{2(3-|a_0|^2)}}+ \frac{3}{16}(1-|a_0|^2)^2 \\
&=& 1- \frac{1-|a_0|^2}{16}\left (13+3|a_0|^2 - \frac{16}{\sqrt{8}}\frac{1}{\sqrt{1-|a_0|^2}}-\frac{16}{\sqrt{2(3-|a_0|^2)}}\right ),
\end{eqnarray*}
which is less than $1$, because for $0\leq x<1/3$, we have
$$ \Phi_3(x)=13+3x^2 - \frac{16}{\sqrt{8}}\frac{1}{\sqrt{1-x^2}}-\frac{16}{\sqrt{2(3-x^2)}} > 0.
$$
Here we have used the fact that, for $0\leq x<1/3$,
$$13+3x^2 \geq 13> 16 \left(\frac{3}{8}\right) +16\left(\frac{3}{\sqrt{52}}\right) > \frac{16}{\sqrt{8}}\frac{1}{\sqrt{1-x^2}}+\frac{16}{\sqrt{2(3-x^2)}}.
$$
Combining the two cases shows that $H_2(r)\leq 1 $ for $r\leq 1/3$.

Concerning the sharpness part, we just need to consider the function
$f_0=h_0+\overline{g_0}$, where $h_0$ and $g_0$ are defined by \eqref{KayPon8-eq7c},
and follow a similar procedure as before. This completes the proof.
\end{proof}

\begin{thm}
Suppose that $f(z)=h(z)+\overline{g(z)}=\sum_{k=0}^\infty a_k z^k+\overline{\sum_{k=1}^\infty b_k z^k}$ is a harmonic mapping in $\mathbb{D}$, with $\|h\|_{\infty} = 1$ and $|g'(z)| \leq |h'(z)|$  for $z\in\ID$. Then
$$L(r):=|a_0| + \sum_{k=1}^{\infty} \big(|a_k|+|b_k|\big)r^k + \frac{3}{8}\sum_{k=1}^{\infty} \big(|a_k|^2 +|b_k|^2 \big)r^k
\leq 1 ~\mbox{ for  }~ r\leq \frac{1}{5}.
$$
The constants $3/8$ and $1/5$ cannot be improved.
\end{thm}
\begin{proof}
Note that $L$ is an increasing function of $r ~(>0)$ and hence, we only need to prove the desired inequality for $r=1/5$.
For the function $h$, the inequality \eqref{KP2-eq3-b} holds
and therefore, since $|g'(z)| \leq |h'(z)|$ for $z\in\ID$, by Lemma \Ref{KayPon9-thm2}, we have
\begin{eqnarray*} 
\sum_{k=1}^{\infty} |a_k|^2r^k +\sum_{k=1}^{\infty} |b_k|^2r^k  \leq 2r\frac{(1-|a_0|^2)^2}{1-r|a_0|^2} ~\mbox{ for $0 < r \leq 1/2$.}
\end{eqnarray*}
We divide the proof into two parts. Firstly, we consider the case where $|a_0|\geq 1/5$. By hypothesis (see Equation \eqref{KP2-eq3-b1}), we have
$$\sum_{k=1}^\infty |b_k|r^k \leq  C(r):=\frac{(1-|a_0|^2)r}{\sqrt{(1-|a_0|^2r)(1-r)}} ~\mbox{ for }~ r\leq 1/2.
$$
Therefore, for $r\leq 1/5$, by \eqref{KayPon8-eq7} and the above two inequalities, we find that
\begin{eqnarray*}
L(r) &\leq & |a_0| + A(1/5) + C(1/5) + \frac{3}{4} \left( \frac{(1-|a_0|^2)^2}{5-|a_0|^2}\right)\\
 &=& |a_0| + \frac{1-|a_0|^2}{5-|a_0|} + \frac{1-|a_0|^2}{2\sqrt{5-|a_0|^2}}+ \frac{3}{4} \left( \frac{(1-|a_0|^2)^2}{5-|a_0|^2}\right)\\
 &=& 1- \frac{1-|a_0|}{4(5-|a_0|)(5-|a_0|^2)}\Phi (|a_0|),
\end{eqnarray*}
where $\Phi (x) = -3x^4+20x^3 +2x^2-52x+65+ (2x^2-8x-10)\sqrt{5-x^2}.$
It follows easily that
$$\Phi '(x) = -12x^3+60x^2+4x-52 + \frac{-6x^3 + 16x^2+30x-40}{\sqrt{5-x^2}}
$$
which is negative on $[1/5,1]$. Indeed, if
$$\Psi _1(x) = -12x^3+60x^2+4x-52 \mbox{ and } \Psi _2(x)= -6x^3 + 16x^2+30x-40,
$$
then it is clear that $\Psi_1'(x)>0$ and $\Psi_2'(x)>0$ on $[1/5,1]$ and thus,
$\Psi_1(x) \leq \Psi_1(1)=0$ and $\Psi_2(x) \leq \Psi_2(1)=0$ on $[1/5,1]$. Thus, $\Phi $ is decreasing on $[1/5,1]$ so that
$\Phi(x) \geq \Phi (1)=0$ on $[1/5,1]$. Hence, for $|a_0|\geq 1/5$, we obtain that $L(r)\leq 1$ for $r\leq 1/5$.

Next, for $|a_0|<1/5$, we have
\begin{eqnarray*}
L(r) &\leq & |a_0| + B(1/5) + C(1/5) + \frac{3}{4} \left( \frac{(1-|a_0|^2)^2}{5-|a_0|^2}\right)\\
&=& |a_0| + \sqrt{\frac{1-|a_0|^2}{24}} + \frac{1-|a_0|^2}{2\sqrt{5-|a_0|^2}} + \frac{3}{4} \left( \frac{(1-|a_0|^2)^2}{5-|a_0|^2}\right)\\
&<& \frac{1}{5} + \frac{\sqrt{6}}{12} + \frac{1}{4} + \frac{3}{16}<1,
\end{eqnarray*}
which shows that $L(r)\leq 1$ for $r\leq 1/5$.



In order to prove the sharpness, we may consider the same function
$f_0=h_0+\overline{g_0}$, where $h_0$ and $g_0$ are defined by \eqref{KayPon8-eq7c} with $|\lambda|=1$. Then,
with
$$L_c(r)=|a_0| + \sum_{k=1}^{\infty} \big(|a_k|+|b_k|\big)r^k +  2c\sum_{k=1}^{\infty} \big(|a_k|^2 +|b_k|^2 \big)r^k
$$
we find for $f_0=h_0+\overline{g_0}$ that
\begin{eqnarray*}
L_c(1/5) &=& |a| + 2\left(\frac{1-|a|^2}{5-|a|}\right) + 2c \left( \frac{(1-|a|^2)^2}{5-|a|^2}\right) \\
&=& 1- (1-|a|)\left[1-2\left(\frac{1+|a|}{5-|a|}\right) - 2c \left( \frac{(1-|a|^2)(1+|a|)}{5-|a|^2} \right) \right] .
\end{eqnarray*}
Note that $L_c(1/5)=L(1/5)$ if $c=3/8$. However, for  $c>3/8$, the quantity in the square bracket term on the right is not only
less than the corresponding quantity for $c=3/8$,
but it also approaches $0$ when $|a|$ tends to $1$. Thus, for the function $f_0=h_0+\overline{g_0}$, $L(1/5)$ is greater than $1$
when $|a|$ is sufficiently close to $1$, $|\lambda|=1$ and  $c>3/8$. This proves the sharpness.
\end{proof}

\begin{thm}
Suppose that $f(z)=h(z)+\overline{g(z)}=\sum_{k=0}^\infty a_k z^k+\overline{\sum_{k=1}^\infty b_k z^k}$ is a harmonic mapping in $\mathbb{D}$, where $\|h\|_{\infty} = 1$ and $|g'(z)| \leq |h'(z)|$  for $z\in\ID$. Then
$$N(r):=|a_0| + \sum_{k=1}^\infty (|a_k|+|b_k|)r^k + |h(z)-a_0|^2 \leq 1  ~\mbox{ for  }~ r\leq 1/5.
$$
The constant $1/5$ is best possible.
\end{thm}
\begin{proof}
By using the triangle inequality we can see that
\begin{eqnarray}
N(r) \leq |a_0| + \sum_{k=1}^\infty (|a_k|+|b_k|)r^k + \left( \sum_{k=1}^\infty |a_k|r^k\right)^2 .\label{Th3_eq1}
\end{eqnarray}
The right hand side of \eqref{Th3_eq1} is an increasing function of $r$ and thus it is enough to prove the required inequality for $r=1/5$.
Again, we apply the same method of proof as in the previous two theorems.

Firstly, we consider the case where $|a_0|\geq 1/5$. We see that
\begin{eqnarray*}
N(1/5) &\leq &  |a_0| + A(1/5) + C(1/5) +  [A(1/5)]^2 \\
&=& |a_0| + \frac{1-|a_0|^2}{5-|a_0|} + \frac{1-|a_0|^2}{2\sqrt{5-|a_0|^2}} + \frac{(1-|a_0|^2)^2}{(5-|a_0|)^2}\\
&=&1-\frac{1-|a_0|}{2(5-|a_0|)^2\sqrt{5-|a_0|^2}} \Phi (|a_0|),
\end{eqnarray*}
where
$$\Phi (x) = -x^3+9x^2-15x-25 + 2(x^3+3x^2-15x+19)\sqrt{5-x^2}.
$$
In order to show that $\Phi (x)\geq 0$ for $x\in [0,1]$, we compute that
\beqq
\Phi '(x) &= & -3x^2+18x-15 +2\left (\frac{-4x^4-9x^3+45x^2+11x-75}{\sqrt{5-x^2}}\right )\\
&=& -3(x-1)(x-5)- 2\left (\frac{4x^4+9x^3+19+45(1-x^2)+11(1-x)}{\sqrt{5-x^2}}\right ),
\eeqq
which implies that $\Phi $ is decreasing on $[1/5,1]$ and thus, $\Phi(x)\geq \Phi(1)=0$. Hence, $N(1/5)\leq  1$ whenever
$|a_0|\geq 1/5$.

Secondly, for $|a_0|<1/5$, we have
\begin{eqnarray*}
N(1/5) &\leq & |a_0| + B(1/5) + C(1/5) +  [B(1/5)]^2 \\
&=& |a_0| + \sqrt{\frac{1-|a_0|^2}{24}} + \frac{1-|a_0|^2}{2\sqrt{5-|a_0|^2}} + \frac{1-|a_0|^2}{24} \\
&<& \frac{1}{5}+\frac{\sqrt{6}}{12} + \frac{1}{4} + \frac{1}{24}<1.
\end{eqnarray*}
For the sharpness,  we may consider the same function
$f_0=h_0+\overline{g_0}$, where $h_0$ and $g_0$ are defined by \eqref{KayPon8-eq7c} with $|\lambda|=1$.
Then, for the harmonic mapping $f_0$, we see that
$$N(r)= |a| + 2r \left(\frac{1-|a|^2}{1-|a|r} \right) + r^2\left( \frac{(1-|a|^2)^2}{|1-\overline{a}z|}\right).$$
We choose $z \in \mathbb{D}$ such that $\overline{a}z \in \mathbb{R}$ and it follows that $N(r)>1$ if and only if $1/5<r<1$, which completes the proof.
\end{proof}

Now, we state and prove the harmonic analog of Theorem \Ref{KayPon8-1to3}(3). Moreover, it is also possible to derive a version of the
next theorem by replacing the condition ``$h$ is a bounded function in $\ID$ such that $\|h\|_\infty =1$"  by
``$f$ is a bounded function in $\ID$ such that $\|f\|_\infty =1$."

\begin{thm}
Suppose that $f(z) = h(z)+\overline{g(z)}=\sum_{k=0}^\infty a_k z^k+\overline{\sum_{k=1}^\infty b_k z^k}$ is a
harmonic mapping of the disk $\ID$, where $h$ is a bounded function in $\ID$ such that $\|h\|_\infty =1$ and $|g'(z)|\leq |h'(z)|$ for $z\in\ID$.
Then,
$$|h(z)|^2 + \sum_{k=1}^\infty (|a_k|^2+ |b_k|^2)r^{2k} \leq 1 ~\mbox{ for  }~ r \leq r_0
$$
where $r_0=\sqrt{5/(9+4\sqrt{5})} ~\approx 0.527864$ is the unique root positive root of the equation
$$(10+6r^2)^{3/2} + 144r^2 -80=0
$$
in the interval $(0,1/\sqrt{2})$. The number $r_0$ is the best possible.
\end{thm}
\begin{proof}
By using the Schwarz-Pick lemma for the function $h$, we have
$$|h(z)|\leq \frac{r+|a_0|}{1+r|a_0|}, \quad |z|=r.
$$
By assumption both \eqref{KP9-eq6} and \eqref{KP2-eq3-b} hold for $0 < r \leq 1/2$. In particular, for $0 < r \leq 1/\sqrt{2}$,  we can easily obtain
\begin{eqnarray*}
|h(z)|^2 + \sum_{k=1}^\infty (|a_k|^2+ |b_k|^2)r^{2k}
&\leq & \Big(\frac{r+|a_0|}{1+r|a_0|}\Big)^2 + \frac{2r^2(1-|a_0|^2)^2}{1-|a_0|^2r^2} \\
&=& 1- \left [\frac{(1+r|a_0|)^2-(r+|a_0|)^2}{(1+r|a_0|)^2} -\frac{2r^2(1-|a_0|^2)^2}{1-|a_0|^2r^2}\right ]\\
&=& 1- \left [\frac{(1-|a_0|^2)(1-r^2)}{(1+r|a_0|)^2} -\frac{2r^2(1-|a_0|^2)^2}{1-|a_0|^2r^2}\right ]\\
&=& 1- \frac{(1-|a_0|^2)}{(1+r|a_0|)^2(1-|a_0|r)}\Phi(|a_0|,r),
\end{eqnarray*}
where
$$\Phi(x,r) = 2r^3x^3+2r^2x^2 -(r+r^3)x +1-3r^2,
$$
with $x=|a_0|\leq 1$ and $0 < r \leq 1/\sqrt{2}$. We wish to show that $\Phi(x,r)\geq 0$ for all $x\in [0,1]$ and $r\in [0,r_0]$.
We notice that $r_0<1/\sqrt{3}<1/\sqrt{2}$ and partial derivative with respect to $x$ gives
$$\Phi_x(x,r) = 6r^3x^2+4r^2x -r(1+r^2).
$$
Also, we observe that $\Phi_{xx}(x,r)>0$ for  $x,r>0$. Moreover, by determining the roots of the equation $\Phi_x(x,r) =0$, we may write
$$\Phi_x(x,r) = 6r^3(x-x_{-})(x-x_+), \quad x_{\pm}=\frac{-2\pm \sqrt{10+6r^2}}{6r},
$$
which shows that $\Phi(x,r)$ is a decreasing function of $x$ on $(0,x_+)$ and an increasing function of $x$  on $(x_+,1)$ where $x_+$ is the only point of
local minimum on $(0,1)$ for each fixed $r$. In particular, we have
$$\Phi(x,r) \geq \Phi(x_+,r) ~\mbox{ for $x\in [0,1]$ and $0\leq r\leq 1/\sqrt{2}$} .
$$
Next we observe that $ \Phi(0,r)=1-3r^2>0$ for $0 < r \leq 1/\sqrt{2}$. Clearly, to complete the proof and to determine the range of $r$, it suffices to show that
$\Phi(x_+,r)\geq 0$ for $r\leq r_0$. In order to do this, we set $y=\sqrt{10+6r^2}$ (so that $r^2= (y^2-10)/6$ and $1+r^2=(y^2-4)/6$) and find that
$$\Phi(x_+,r) = \frac{2}{6^3}(y-2)^3 + \frac{2}{6^2}(y-2)^2 -\frac{1+r^2}{6}(y-2)+1-3r^2
$$
which after simplification gives that
$$\Phi(x_+,r) =- \frac{1}{54}[y^3 + 24y^2 -320]= - \frac{1}{54}[(10+6r^2)^{3/2} + 24(10+6r^2) -320].
$$
We deduce that $\Phi(x_+,r)$ is nonnegative for $r\in [0,r_0]$, where $r_0=\sqrt{5/(9+4\sqrt{5})} $ is the unique positive root of the equation
$\Phi(x_+,r_0)=0$ in the interval $(0,1/\sqrt{2})$. Following the procedure adopted for the sharpness of the previous theorems, we consider the function $f_0=h_0+\overline{g_0}$, where
$$ h_0(z) =\frac{z+a}{1+az}, \quad a=\frac{-2+\sqrt{10+6r_0^2}}{6r_0},
$$
and $g=\lambda h$ ($|\lambda|=1$).  Rest of the details is omitted and the proof is complete.
\end{proof}

\section{Concluding Remarks}
There are interesting examples of harmonic mappings $f = h+\overline{g}$ that are of the form  $g'(z)=\eta zh'(z)$ in $\mathbb{D}$ for some $\eta$ with $|\eta|=1$
(cf. \cite{Bshouty-Lyzzaik-2010,PonSai2015}). Such mappings play a significant role in the theory of harmonic mappings and sometimes with some
additional conditions on $h$. For example, $f(z)=z+\frac{1}{2}\overline{z^2}$ is extremal for the area minimizing property of harmonic mappings
in ${\mathcal S}_H^0$. Here ${\mathcal S}_H^0$ denotes the class of all sense-preserving univalent harmonic mappings in $\ID$ of the form
$f(z) =z+\sum_{k=2}^\infty a_k z^k+\overline{\sum_{k=2}^\infty b_k z^k}$. Again the half-plane mapping in ${\mathcal S}_H^0$, e.g. the extremal function
for the coefficient estimates of the family of convex mappings in ${\mathcal S}_H^0$, and harmonic Koebe function in ${\mathcal S}_H^0$, satisfy the
$g'(z)=\eta zh'(z)$ when  $\eta=\pm 1$ (cf. \cite{CS}). Hence, the following propositions could be regarded as some useful observations. Naturally, the Bohr radius
in this case should clearly be greater than or equal to $1/5$.

\begin{prop}\label{prop1}
Suppose that $f(z) = h(z)+\overline{g(z)}=\sum_{k=0}^\infty a_k z^k+\overline{\sum_{k=1}^\infty b_k z^k}$ is a
harmonic mapping of the disk $\ID$, where $h$ is a bounded function in $\ID$ such that $|h(z)|<1$ and $|g'(z)|= |zh'(z)|$ for $z\in\ID$.
Then
$$
B_1(r):=|a_0|+\sum_{k=1}^\infty (|a_k| +|b_k|)r^k  \leq 1 ~\mbox{ for  }~ r \leq r_0, 
$$
where $r_0\approx 0.299824$ is the solution of the equation
\begin{equation}\label{HEq_Th3b}
5x + 2(1-x)\log(1-x)=1.
\end{equation}
The constant $r_0$ is the best possible.
\end{prop}
\begin{proof}
By assumption,  there exists an $\eta$ such that $|\eta|=1$ and $g'(z)=\eta zh'(z)$ in $\mathbb{D}$. Note that $b_1=g'(0)=0$. Then the comparison
of the coefficients of $h$ and $g$ shows that
$$b_k=\eta\left (\frac{k-1}{k}\right ) a_{k-1} ~\mbox{ for  }~k\geq 2,
$$
where $|a_k|\leq 1-|a_0|^2$ for $k\geq 1$ and therefore (since $b_1=0$)
\begin{eqnarray*}
B_1(r)
&\leq & |a_0| + (1-|a_0|^2)\frac{r}{1-r} + \sum_{k=2}^\infty \frac{k-1}{k}|a_{k-1}|r^k \\
&\leq & |a_0| + (1-|a_0|^2)\frac{r}{1-r} +(1-|a_0|^2) \sum_{k=1}^\infty \frac{k}{k+1}r^{k+1}\\
&=& |a_0|  +(1-|a_0|^2) \left( \frac{2r}{1-r} + \log(1-r)\right )\\
&\leq & |a_0| + (1-|a_0|^2)\frac{1}{2}, \quad \mbox{ by \eqref{HEq_Th3b} and $r\leq r_0$},\\
&= & \frac{1}{2}[2- (1-|a_0|)^2]\leq 1.
\end{eqnarray*}

For the sharpness of this result,  consider
$$h_0(z)= \frac{a+z}{1+\overline{a} z}=a +(|a|^2-1) \sum_{k=1}^\infty \overline{a}^{k-1}z^k \mbox{ and }
g_0(z) = (|a|^2-1) \sum_{k=1}^\infty e^{i\theta }\frac{k-1}{k} \overline{a}^{k-2}z^k ,
$$
where $a\in \ID$. Then a computation on the above choices of $h_0$ and $g_0$ gives that
$$|a_0| + \sum_{k=1}^\infty ( |a_k| + |b_k|) r^k =
1- (1 - |a|)\left[1-  (1+|a|) \left(\frac{r^2 + r}{1-|a|r}+ \frac{r}{|a|} + \frac{\log(1-|a|r)}{|a|^2} \right ) \right]
$$
in which the quantity in the squared bracket term on the right becomes negative when $|a|$ approaches $1$ and $r>r_0$.
\end{proof}

Also, the above proposition can be naturally improved as follows.

\begin{prop}
Assume the hypotheses of Proposition \ref{prop1}.
Then
\begin{equation}\label{HEq_Th3c}
H_1(r):=|a_0|+\sum_{k=1}^\infty (|a_k| +|b_k|)r^k+ K\left (\frac{S_r}{\pi}\right )  \leq 1 ~\mbox{ for  }~ r \leq r_0,
\end{equation}
where $r_0\approx 0.299824$ is  defined as in Proposition \ref{prop1},
$$K=\frac{1}{8}\left[ \frac{2r_0^2}{1-r_0^2} + \log(1-r_0^2)\right ]^{-1} \approx 1.209452,
$$
and $S_r$ denotes the same quantity as defined before. The values of $K$ and $r_0$ are sharp.
\end{prop}
\begin{proof}
For the area $S_r$, following the notation and the proof of Proposition \ref{prop1}, we have
\begin{eqnarray}
\frac{S_r}{\pi} &=&\sum_{k=1}^\infty k|a_k|^2r^{2k} - \sum_{k=2}^\infty  k|b_k|^2r^{2k} \nonumber\\
&=& \sum_{k=1}^\infty k\left (1-\frac{kr^2}{k+1}\right )|a_k|^2r^{2k}  \nonumber\\
&\leq & (1-|a_0|^2)^2\sum_{k=1}^\infty k\left (1-\frac{kr^2}{k+1}\right )r^{2k} \nonumber\\
&=& (1-|a_0|^2)^2 \left [ \frac{r^2}{(1-r^2)^2}- \int_0^{r^2} \frac{t(1-t)}{(1-t)^3}dt\right ]. \label{KP9-eq7}
\end{eqnarray}
We calculate the integral by using the formula $t+t^2=2-3(1-t)+(1-t)^2$, and obtain
$$\int_0^{r^2} \frac{t(1-t)}{(1-t)^3}dt= \frac{2r^2-r^4}{(1-r^2)^2}-\frac{3r^2}{1-r^2} -\log(1-r^2).
$$
Finally, using the last integral and \eqref{KP9-eq7}, we deduce that
$$\frac{S_r}{\pi} \leq (1-|a_0|^2)^2 \left [ \frac{2r^2}{1-r^2}+ \log(1-r^2) \right ].
$$
Now, as $H_1$ defined by \eqref{HEq_Th3c} is an increasing function of $r$, we have for $r\leq r_0$
\begin{eqnarray*}
H_1(r) &\leq & |a_0| + \sum_{k=1}^\infty (|a_k|+|b_k|)r^k  +  K\left( \frac{S_r}{\pi} \right )\\
&\leq & |a_0| + \left ( \frac{2r_0}{1-r_0} + \log(1-r_0)\right )(1-|a_0|^2) \\
&& \hspace{1cm}   + K\left ( \frac{2r_0^2}{1-r_0^2} + \log(1-r_0^2)\right )(1-|a_0|^2)^2 \\
&=& 1- \frac{(1-|a_0|)^2}{2}\left[ 1-\frac{(1+|a_0|)^2}{4}\right ] \\
&\leq &1.
\end{eqnarray*}
Sharpness may be proved similarly. We omit the details.
\end{proof}

\subsection*{Acknowledgments}
This work was done during second author's research visit to Aalto University, Finland, which received support from the Aalto University Science Institute (AScI) Visiting Fellowship Programme. The research was supported by the Academy of Finland.

\end{document}